\documentclass[12pt]{article}
\usepackage[]{amsmath,amssymb}
\usepackage{amscd}
\usepackage{latexsym}
\usepackage{cite}
\usepackage{amsthm}

\newtheorem{definition}{Definition}[section]
\newtheorem{theorem}[definition]{Theorem}
\newtheorem{lemma}[definition]{Lemma}
\newtheorem{corollary}[definition]{Corollary}

\newtheorem{example}[definition]{Example}
\newtheorem{conjecture}[definition]{Conjecture}

\newtheorem{note}[definition]{Note}

\newtheorem{proposition}[definition]{Proposition}

\def\Z{\mathbb Z}
\def\F{\mathbb F}

\begin{document}
\title{\bf  
Tridiagonal pairs of $q$-Racah type, \\
the Bockting
operator $\psi$,
and \\
$L$-operators for 
$U_q(L({\mathfrak{sl}}_2))$
}
\author{
Paul Terwilliger}
\date{}

\maketitle
\begin{abstract}
We describe the Bockting operator $\psi$
for a tridiagonal pair of $q$-Racah type,
in terms of a certain $L$-operator for the quantum loop
algebra $U_q(L({\mathfrak {sl}}_2))$. 
\bigskip

 \end{abstract}

\section{Introduction}
\noindent 
In the theory of quantum groups there exists the concept of
an 
$L$-operator; this
was introduced
in \cite{ks}  to 
 obtain solutions
for the Yang-Baxter equation.
In linear algebra there exists the concept of a tridiagonal pair;
this was introduced in
\cite{TD00}  to describe the irreducible modules
for the subconstituent algebra of a $Q$-polynomial distance-regular
graph. Recently some authors have connected the two concepts.
In 
\cite{bas2},
\cite{BK05}
Pascal Baseilhac and Kozo Koizumi use 
$L$-operators for the quantum loop algebra
$U_q(L(\mathfrak{sl}_2))$  to  construct a family
of finite-dimensional modules for the   
$q$-Onsager algebra
${\mathcal O}_q$; 
see
\cite{bas4,
basnc,
basXXZ,
 bas7}
for related work.
A finite-dimensional irreducible ${\mathcal O}_q$-module 
is essentially the same thing as a tridiagonal pair of
$q$-Racah type
\cite[Section~12]{bocktingTer},
\cite[Section~3]{TwoRel}.
In \cite[Section~9]{miki}, Kei Miki uses similar $L$-operators 
to describe
how
 $U_q(L(\mathfrak{sl}_2))$
is related to 
the $q$-tetrahedron algebra $\boxtimes_q$.
A finite-dimensional irreducible $\boxtimes_q$-module 
is essentially the same thing as a tridiagonal pair of $q$-geometric
type 
\cite[Theorem~2.7]{nonnil},
\cite[Theorems~10.3, 10.4]{qtet}.
Following Baseilhac, Koizumi, and Miki, in the present paper we use  $L$-operators
for 
 $U_q(L(\mathfrak{sl}_2))$
to describe the 
 Bockting operator $\psi$
associated with a tridiagonal pair of $q$-Racah type.
Before going into detail, we recall some notation and basic concepts.
Throughout this paper $\mathbb F$ denotes a field.
Let $V$ denote a vector space over $\mathbb F$ with finite
positive dimension. 
For an $\mathbb F$-linear map $A:V\to V$ and a subspace
$W \subseteq V$, we say that $W$ is an {\it eigenspace} of $A$
whenever $W \not=0$ and there exists $\theta \in \mathbb F$ such that
$W = \lbrace v \in V|Av=\theta v\rbrace$;  in this case
$\theta$ is called the {\it eigenvalue} of $A$ associated with $W$.
We say that $A$ is {\it diagonalizable} whenever $V$ is spanned
by the eigenspaces of $A$.
\begin{definition} 
\label{def:TDpair}
\rm (See \cite[Definition~1.1]{TD00}.)  
Let $V$ denote a vector space over $\mathbb F$ with finite
positive dimension. 
By a {\it tridiagonal pair} (or {\it TD pair}) on $V$ we mean
an ordered pair of $\mathbb F$-linear maps $A:V\to V$ and $A^*:V\to V$
that satisfy the following four conditions:
\begin{enumerate}
\item[\rm (i)] Each of $A,A^*$ is diagonalizable.
\item[\rm (ii)] There exists an ordering 
$\lbrace V_i\rbrace_{i=0}^d$ of the eigenspaces of $A$ such that
\begin{eqnarray}
\label{eq:AsAct}
A^* V_i \subseteq V_{i-1} + V_i + V_{i+1} \qquad \qquad (0 \leq i \leq d),
\end{eqnarray}
where $V_{-1} = 0$ and $V_{d+1} = 0$.
\item[\rm (iii)] There exists an ordering 
$\lbrace V^*_i\rbrace_{i=0}^\delta$ of the eigenspaces of $A^*$ such that
\begin{eqnarray}
\label{eq:AAct}
A V^*_i \subseteq V^*_{i-1} + V^*_i + V^*_{i+1} \qquad \qquad
(0 \leq i \leq \delta),
\end{eqnarray}
where $V^*_{-1} = 0$ and $V^*_{\delta+1} = 0$.
\item[\rm (iv)] There does not exist a subspace $W \subseteq V$ such
that
$A W \subseteq W$,
$A^* W \subseteq W$,
$W \not=0$,
$W \not=V$.
\end{enumerate}
\end{definition}
\noindent
We refer the reader to
\cite{TDclass,
TD00,
ITaug}
for background on TD pairs, and
here mention only a few essential points.
Let $A,A^*$ denote a TD pair on $V$, as in Definition
\ref{def:TDpair}.  By 
\cite[Lemma~4.5]{TD00}
the integers  $d$ and $\delta$ from
(\ref{eq:AsAct}) and
(\ref{eq:AAct}) are equal; we call this common value the
{\it diameter} of $A,A^*$.
An ordering of the eigenspaces for $A$ (resp. $A^*$) is
called {\it standard} whenever it satisfies
(\ref{eq:AsAct}) 
(resp. (\ref{eq:AAct})). Let
$\lbrace V_i \rbrace_{i=0}^d$ denote a standard ordering
of the eigenspaces of $A$.
By \cite[Lemma~2.4]{TD00}
the ordering
$\lbrace V_{d-i} \rbrace_{i=0}^d$ is standard and no further
ordering is standard. A similar result holds for the
eigenspaces of $A^*$.
Until the end of this section
fix a standard ordering
$\lbrace V_i \rbrace_{i=0}^d$ 
(resp. 
$\lbrace V^*_i \rbrace_{i=0}^d$) 
of the eigenspaces for $A$ (resp. $A^*$).
For $0 \leq i \leq d$ let $\theta_i $ (resp. $\theta^*_i$)
denote the eigenvalue of $A$ (resp. $A^*$) for the eigenspace
$V_i$ (resp. $V^*_i$). By construction
$\lbrace \theta_i \rbrace_{i=0}^d$ are mutually
distinct and contained in $\mathbb F$. Moreover
$\lbrace \theta^*_i \rbrace_{i=0}^d$ are mutually
distinct and contained in $\mathbb F$. By 
\cite[Theorem~11.1]{TD00}
the expressions
\begin{eqnarray*}
\frac{\theta_{i-2}-\theta_{i+1}}{\theta_{i-1}-\theta_i}
\qquad \qquad 
\frac{\theta^*_{i-2}-\theta^*_{i+1}}{\theta^*_{i-1}-\theta^*_i}
\end{eqnarray*}
are equal and independent of $i$ for $2 \leq i \leq d-1$.
For this constraint the solutions can be given in closed form
\cite[Theorem~11.2]{TD00}.
The ``most general'' solution is called $q$-Racah,
and will be described
shortly.
\medskip

\noindent We now recall the split decomposition 
\cite[Section~4]{TD00}.
 For $0 \leq i \leq d$
define
\begin{eqnarray*}
U_i = (V^*_0 + V^*_1 + \cdots + V^*_i)\cap (V_0 + V_1 + \cdots + V_{d-i}).
\end{eqnarray*}
For notational convenience define $U_{-1}=0$ and $U_{d+1}=0$.
By
\cite[Theorem~4.6]{TD00}
 the sum
$V = \sum_{i=0}^d U_i$ is direct. 
By
\cite[Theorem~4.6]{TD00}
 both
\begin{eqnarray*}
&&
U_0 + U_1 + \cdots + U_i = V^*_0 + V^*_1 + \cdots + V^*_i,
\\
&&
U_i + U_{i+1} + \cdots + U_d = V_0 + V_1 + \cdots + V_{d-i}
\end{eqnarray*}
for $0 \leq i \leq d$. By 
\cite[Theorem~4.6]{TD00}
 both
\begin{eqnarray}
(A-\theta_{d-i} I) U_i \subseteq U_{i+1},
\qquad \qquad 
(A^*-\theta^*_i I) U_i \subseteq U_{i-1}
\label{eq:AActU}
\end{eqnarray}
for $0 \leq i \leq d$.
\medskip

\noindent 
We now describe the $q$-Racah case.
Pick a nonzero $q \in \mathbb F$ such that
$q^4 \not=1$. 
We say that $A,A^*$ has {\it $q$-Racah type} whenever
there exist nonzero $a,b \in \mathbb F$ such that  both
\begin{eqnarray}
\theta_i = a q^{2i-d} + a^{-1} q^{d-2i}, 
\qquad \qquad
\theta^*_i = b q^{2i-d} + b^{-1} q^{d-2i} 
\label{eq:qracah}
\end{eqnarray}
for $0 \leq i \leq d$. For the rest of this section assume
that $A,A^*$ has $q$-Racah type. For $1 \leq i \leq d$
we have $q^{2i} \not=1$; otherwise $\theta_i = \theta_0$.
Define an $\mathbb F$-linear map
$K: V\to V$ such that for $0\leq i \leq d$, $U_i$ is an
eigenspace of $K$ with eigenvalue $q^{d-2i}$. Thus
\begin{eqnarray}
(K-q^{d-2i}I)U_i = 0 \qquad \qquad (0 \leq i \leq d),
\label{eq:K}
\end{eqnarray}
where $I:V\to V$ is the identity map.
Note that $K$ is invertible.
For $0 \leq i \leq d$ the following holds on $U_i$:
\begin{eqnarray}
\label{eq:thdmi}
a K + a^{-1} K^{-1} = \theta_{d-i} I.
\end{eqnarray}
Define an $\mathbb F$-linear map $R:V\to V$ such that
for $0 \leq i \leq d$, $R$ acts on $U_i$ as $A-\theta_{d-i}I$.
By (\ref{eq:thdmi}),
\begin{eqnarray}
\label{eq:AKR}
A = a K + a^{-1} K^{-1} + R.
\end{eqnarray}
By the equation on the left in (\ref{eq:AActU}),
\begin{eqnarray}
R U_i \subseteq U_{i+1} \qquad \qquad (0 \leq i \leq d).
\label{eq:Raise}
\end{eqnarray}
\noindent We now recall the Bockting operator $\psi$.
By 
\cite[Lemma~5.7]{bockting}
there exists a unique $\mathbb F$-linear map
$\psi: V\to V$ such that both
\begin{eqnarray}
&&\psi U_i \subseteq U_{i-1} \qquad \qquad (0 \leq i \leq d),
\label{eq:psiActU}
\\
&&\psi R - R \psi = (q-q^{-1})(K-K^{-1}).
\label{eq:psiVSR}
\end{eqnarray}
The known properties of $\psi$ are described in 
\cite{bockting1,
bockting,
bocktingTer}. 
Suppose we are given 
$A,A^*,R,K$ in matrix form, and wish to obtain
$\psi$ in matrix form. This can be
done using
(\ref{eq:Raise}),
(\ref{eq:psiActU}), 
(\ref{eq:psiVSR}) and induction on $i$.
The calculation can be tedious, so one desires
a more explicit description of
$\psi$.
In the present paper we give an explicit description
of 
 $\psi$, 
in terms of a certain  $L$-operator
for 
$U_q(L({\mathfrak {sl}}_2))$. According to this description, 
$\psi$ is equal to $-a$ times the ratio of two components for 
the $L$-operator.
Theorem 
\ref{thm:main}
is our main result.
\medskip

\noindent The paper is organized as follows.
In Section 2 we review the 
algebra
$U_q(L({\mathfrak {sl}}_2))$ in its Chevalley presentation.
In Section 3 we recall the
equitable presentation for
$U_q(L({\mathfrak {sl}}_2))$.
In Section 4
we discuss some $L$-operators for
$U_q(L({\mathfrak {sl}}_2))$.
In Section 5 we use these $L$-operators to describe $\psi$.

\section{The quantum loop algebra
$U_q(L({\mathfrak {sl}}_2))$}
Recall the integers $\mathbb Z=\lbrace 0, \pm 1, \pm 2, \ldots\rbrace$
and natural numbers
$\mathbb N = \lbrace 0,1,2,\ldots \rbrace$. We will be discussing
algebras. An algebra is meant to be associative and have a 1.
Recall the
field $\mathbb F$.  Until the end of Section 4,
fix a nonzero $q \in \mathbb F$ such that $q^2\not=1$. Define
\begin{eqnarray*}
\lbrack n \rbrack_q = \frac{q^n-q^{-n}}{q-q^{-1}}
\qquad \qquad n \in \mathbb Z.
\end{eqnarray*}
All tensor products are meant to be over $\mathbb F$.

\begin{definition}
\label{def:loop}
\rm (See \cite[Section 3.3]{cp3}.) Let 
$U_q(L({\mathfrak{sl}}_2))$ denote the
$\mathbb F$-algebra with generators $E_i, F_i, K^{\pm 1}_i$ $(i\in \lbrace 0,1\rbrace)$
and  relations
\begin{eqnarray*}
&&K_i K^{-1}_i = 1, \qquad \qquad \qquad K^{-1}_i K_i = 1,
\\
&&
K_0 K_1 = 1,  \qquad \qquad \qquad \quad K_1 K_0 = 1,
\\
&&
K_i E_i = q^2 E_i K_i, \qquad \qquad 
K_i F_i = q^{-2} F_i K_i,
\\
&&
K_i E_j = q^{-2} E_j K_i, \qquad \qquad 
K_i F_j = q^{2} F_j K_i, \qquad i\not=j,
\\
&&
E_i F_j-F_j E_i = \delta_{i,j} \frac{K_i - K_i^{-1}}{q-q^{-1}},
\\
&&
E^3_i E_j
- 
\lbrack 3 \rbrack_q   E^2_i E_j E_i
+
\lbrack 3 \rbrack_q   E_i E_j E^2_i
-
E_j E^3_i = 0,
\qquad \qquad i \not=j,
\\
&&
F^3_i F_j
- 
\lbrack 3 \rbrack_q   F^2_i F_j F_i
+
\lbrack 3 \rbrack_q   F_i F_j F^2_i
-
F_j F^3_i = 0,
\qquad \qquad i \not=j.
\end{eqnarray*}
We call 
$E_i,F_i, K^{\pm 1}_i$ the {\it Chevalley  generators} for
$U_q(L({\mathfrak{sl}}_2))$.
\end{definition}

\begin{lemma} 
\label{lem:loophopf}
{\rm (See \cite[p.~35]{jantzen}.)}
We turn
$U_q(L(\mathfrak{sl}_2))$ into a Hopf algebra as follows.
The coproduct $\Delta$ satisfies
\begin{eqnarray*}
&& \Delta(K_i) = K_i \otimes K_i, \qquad  \qquad \qquad 
\Delta(K^{-1}_i) = K^{-1}_i \otimes K^{-1}_i,
\\
&& \Delta(E_i) = E_i \otimes 1+ K_i \otimes E_i,
\qquad \qquad  \Delta(F_i) = 1 \otimes F_i+ F_i \otimes K^{-1}_i.
\end{eqnarray*}
The counit $\varepsilon$  satisfies
\begin{eqnarray*}
&&\varepsilon(K_i)=1, \qquad 
\varepsilon(K^{-1}_i)=1, \qquad 
\varepsilon(E_i)=0, \qquad 
\varepsilon(F_i)=0.
\end{eqnarray*}
The antipode $S$ satisfies
\begin{eqnarray*}
&& S(K_i)=K^{-1}_i, 
\qquad 
S(K^{-1}_i)=K_i,
\qquad
S(E_i) = -K^{-1}_i E_i,
\qquad
S(F_i)= -F_iK_i.
\end{eqnarray*}
\end{lemma}

\noindent We now discuss the
$U_q(L(\mathfrak{sl}_2))$-modules.

\begin{lemma}
\label{lem:irred}
{\rm (See \cite[Section~4]{cp3}.)}
There exists a family
of 
$U_q(L(\mathfrak{sl}_2))$-modules
\begin{eqnarray}
{\bf V}(d,t) \qquad \qquad 0 \not=d \in \mathbb N, \qquad \qquad
0 \not= t \in \mathbb F
\label{eq:vde}
\end{eqnarray}
with this property: ${\bf V}(d,t)$ has a basis
$\lbrace v_i \rbrace_{i=0}^d$ such that
\begin{eqnarray*}
&&K_1 v_i = q^{d-2i}v_i \qquad \qquad (0 \leq  i \leq d),
\\
&& E_1 v_i =  \lbrack d-i+1 \rbrack_q v_{i-1}
\qquad \qquad (1 \leq i \leq d), \qquad
E_1 v_0 = 0,
\\
&& F_1 v_i = \lbrack i+1 \rbrack_q v_{i+1}
\qquad \qquad (0 \leq i \leq d-1), \qquad
F_1 v_d = 0,
\\
&&K_0 v_i = q^{2i-d}v_i \qquad \qquad (0 \leq  i \leq d),
\\
&& E_0 v_i =  t\lbrack i+1 \rbrack_q v_{i+1}
\qquad \qquad (0 \leq i \leq d-1), \qquad
E_0 v_d = 0,
\\
&& F_0 v_i = t^{-1}\lbrack d-i+1 \rbrack_q v_{i-1}
\qquad \qquad (1 \leq i \leq d), \qquad
F_0 v_0 = 0.
\end{eqnarray*}
The 
module ${\bf V}(d,t)$ is irreducible provided that
$q^{2i} \not=1$ for $1 \leq i \leq d$.
\end{lemma}
\begin{definition}\rm 
Referring to Lemma
\ref{lem:irred},
we call 
${\bf V}(d,t)$
 an {\it evaluation module}
 for 
$U_q(L(\mathfrak{sl}_2))$.
We call $d$ the
{\it diameter}.
We call
$t$ the {\it evaluation parameter}.
\end{definition}

\begin{example}\rm For $0 \not=t \in \mathbb F$ the 
$U_q(L(\mathfrak{sl}_2))$-module
${\bf V}(1,t)$ is described as follows. With
respect to the basis $v_0,v_1$ from Lemma
\ref{lem:irred}, the matrices representing the
Chevalley generators are
\begin{eqnarray*}
&&
E_1:\quad  
\left(
\begin{array}{ c c}
 0 & 1 \\
 0   & 0 
       \end{array} 
       \right),
\qquad \qquad 
F_1:\quad  
\left(
\begin{array}{ c c}
 0 & 0 \\
 1   & 0 
       \end{array}
       \right),
\qquad \qquad 
K_1:\quad  
\left(
\begin{array}{ c c}
 q & 0 \\
0   & q^{-1} 
       \end{array}
\right),
 \\
&&
E_0:\quad  
\left(
\begin{array}{ c c}
 0 & 0 \\
 t   & 0 
       \end{array}
 \right),
\qquad \qquad 
F_0:\quad  
\left(
\begin{array}{ c c}
 0 & t^{-1} \\
 0   & 0
       \end{array}
 \right),
\qquad \qquad 
K_0:\quad  
\left(
\begin{array}{ c c}
 q^{-1} & 0 \\
0   & q
       \end{array}
\right).
\end{eqnarray*}
\end{example}

\begin{lemma} 
\label{lem:UotimesV}
{\rm (See \cite[p.~58]{kassel}.)}
Let $U$ and $V$ denote
$U_q(L(\mathfrak{sl}_2))$-modules.
Then $U\otimes V$ becomes a 
$U_q(L(\mathfrak{sl}_2))$-module as follows.
For $u \in U$ and $v \in V$,
\begin{eqnarray*}
&&
K_i (u \otimes v) = K_i(u) \otimes K_i(v),
\\
&&
K^{-1}_i (u \otimes v) = K^{-1}_i(u) \otimes K^{-1}_i(v),
\\
&&
E_i (u \otimes v) = E_i(u) \otimes v+ K_i(u) \otimes E_i(v),
\\
&&F_i (u \otimes v) = u \otimes F_i( v)+ F_i(u) \otimes K^{-1}_i(v).
\end{eqnarray*}
\end{lemma}

\begin{definition}
\label{def:triv}
\rm (See \cite[p.~110]{cp2}.)
Up to isomorphism, there exists a  unique
$U_q(L(\mathfrak{sl}_2))$-module of dimension 1 on which
each $u \in 
U_q(L(\mathfrak{sl}_2))$ acts as $\varepsilon(u)I$,
where $\varepsilon$ is from
Lemma \ref{lem:loophopf}.
This 
$U_q(L(\mathfrak{sl}_2))$-module is said to be
{\it trivial}.
\end{definition}

\begin{proposition}
\label{lem:CP}
{\rm (See \cite[Theorem~3.2]{miki}.)}
Assume that $\mathbb F$ is algebraically closed
with characteristic zero, and $q$ is not a root of unity.
Let $V$ denote a nontrivial finite-dimensional irreducible
$U_q(L(\mathfrak{sl}_2))$-module on which each eigenvalue of
$K_1$ is an integral power of $q$. 
Then $V$ is isomorphic to a tensor product
of evaluation
$U_q(L(\mathfrak{sl}_2))$-modules.
\end{proposition}

\section{The equitable presentation for 
$U_q(L(\mathfrak{sl}_2))$}

\noindent In this section we recall the
equitable presentation for
$U_q(L(\mathfrak{sl}_2))$.
Let ${\mathbb Z}_4 = {\mathbb Z}/4 {\mathbb Z}$ denote
the  cyclic group of order 4. In a moment we will
discuss some  objects $X_{ij}$. The subscripts
$i,j$ are meant to be in
 ${\mathbb Z}_4$.

\begin{lemma} 
\label{lem:equitLoop}
{\rm
(See
 \cite[Theorem~2.1]{uqsl2hat},
\cite[Proposition~4.2]{miki}.)}
The algebra
$U_q(L(\mathfrak{sl}_2))$ has a presentation
by generators 
\begin{eqnarray}
\label{eq:equitloop}
X_{01}, \quad
X_{12}, \quad
X_{23}, \quad
X_{30}, \quad
X_{13}, \quad
X_{31} 
\end{eqnarray}
and the following relations:
\begin{eqnarray*}
&&X_{13} X_{31} = 1, \quad 
 X_{31} X_{13} = 1,
\quad
\frac{qX_{01}X_{12}-q^{-1}X_{12}X_{01}}{q-q^{-1}}=1,
\quad
\frac{qX_{12}X_{23}-q^{-1}X_{23}X_{12}}{q-q^{-1}}=1,
\\
&&\frac{qX_{23}X_{30}-q^{-1}X_{30}X_{23}}{q-q^{-1}}=1,
\quad 
\frac{qX_{30}X_{01}-q^{-1}X_{01}X_{30}}{q-q^{-1}}=1,
\quad
\frac{qX_{01}X_{13}-q^{-1}X_{13}X_{01}}{q-q^{-1}}=1,
\\
&&\frac{qX_{31}X_{12}-q^{-1}X_{12}X_{31}}{q-q^{-1}}=1,
\quad
\frac{qX_{23}X_{31}-q^{-1}X_{31}X_{23}}{q-q^{-1}}=1,
\quad 
\frac{qX_{13}X_{30}-q^{-1}X_{30}X_{13}}{q-q^{-1}}=1,
\\
&&
 X_{i,i+1}^3 X_{i+2,i+3}
-
\lbrack 3 \rbrack_q
X_{i,i+1}^2 X_{i+2,i+3} X_{i,i+1}
+
\lbrack 3 \rbrack_q
X_{i,i+1} X_{i+2,i+3} X^2_{i,i+1}
-
X_{i+2,i+3} X^3_{i,i+1} =0.
\end{eqnarray*}
An isomorphism with the presentation in
Definition 
\ref{def:loop}
sends
\begin{eqnarray*}
&&X_{01} \mapsto K_0 + q(q-q^{-1})K_0 F_0,
\qquad \qquad
X_{12} \mapsto K_1 - (q-q^{-1})E_1,
\\
&&X_{23} \mapsto K_1 + q(q-q^{-1})K_1 F_1,
\qquad \qquad
X_{30} \mapsto K_0 - (q-q^{-1})E_0,
\\
&&
X_{13} \mapsto K_1,
\qquad \qquad
X_{31} \mapsto K_0.
\end{eqnarray*}
The inverse isomorphism sends
\begin{eqnarray*}
&& E_1 \mapsto (X_{13}-X_{12})(q-q^{-1})^{-1},
 \qquad
 E_0 \mapsto (X_{31}-X_{30})(q-q^{-1})^{-1},
\\
&&
F_1 \mapsto (X_{31} X_{23}-1)q^{-1} (q-q^{-1})^{-1},
\qquad
 F_0 \mapsto (X_{13} X_{01}-1)q^{-1} (q-q^{-1})^{-1},
\\
&&
K_1 \mapsto X_{13},
\qquad \qquad  K_0 \mapsto X_{31}.
\end{eqnarray*}
\end{lemma}

\begin{note}
\label{note:ident}
\rm For notational convenience, we identify
the copy of 
$U_q(L(\mathfrak{sl}_2))$ given in 
Definition
\ref{def:loop}
with the copy given in
Lemma
\ref{lem:equitLoop}, via the isomorphism given in
Lemma
\ref{lem:equitLoop}.
\end{note}

\begin{definition}
\label{def:equitLoopgen}
\rm Referring to Lemma
\ref{lem:equitLoop}, we call the generators
(\ref{eq:equitloop}) the 
{\it equitable generators} for 
$U_q(L(\mathfrak{sl}_2))$.
\end{definition}

\begin{lemma}
{\rm 
(See \cite[Theorem~3.4]{tersym}.)}
From the equitable point
of view the Hopf algebra
$U_q(L(\mathfrak{sl}_2))$ looks as follows. The coproduct 
$\Delta$ satisfies
\begin{eqnarray*}
&&\Delta(X_{13}) = X_{13}\otimes X_{13}, \qquad \qquad
\Delta(X_{31}) = X_{31}\otimes X_{31},
\\
&&\Delta(X_{01})= 
 (X_{01}-X_{31})\otimes 1 + X_{31} \otimes X_{01}, \qquad
\Delta(X_{12})= 
 (X_{12}-X_{13})\otimes 1 + X_{13} \otimes X_{12},
\\
&&\Delta(X_{23})= 
 (X_{23}-X_{13})\otimes 1 + X_{13} \otimes X_{23}, \qquad
\Delta(X_{30})= 
 (X_{30}-X_{31})\otimes 1 + X_{31} \otimes X_{30}.
\end{eqnarray*}
\noindent The counit $\varepsilon$ satisfies
\begin{eqnarray*}
&&
\varepsilon(X_{13})=1,\qquad \qquad 
\varepsilon(X_{31})=1,\qquad \qquad 
\varepsilon(X_{01})=1,
\\
&&\varepsilon(X_{12})=1, \qquad \qquad
\varepsilon(X_{23})=1, \qquad \qquad
\varepsilon(X_{30})=1.
\end{eqnarray*}
\noindent The antipode $S$ satisfies
\begin{eqnarray*}
&&
S(X_{31})=X_{13}, \qquad \qquad
S(X_{13})=X_{31}, \\
&&
S(X_{01}) = 1+  X_{13}- X_{13}X_{01}, \qquad
S(X_{12}) = 1+X_{31}-X_{31}X_{12}, \\
&&
S(X_{23}) = 1+  X_{31}- X_{31}X_{23}, \qquad
S(X_{30}) = 1+X_{13}-X_{13}X_{30}.
\end{eqnarray*}
\end{lemma}

\section{Some $L$-operators for 
$U_q(L(\mathfrak{sl}_2))$}

\noindent 

\noindent 
In this section we recall some $L$-operators for
$U_q(L(\mathfrak{sl}_2))$, and describe their  basic
properties.
\medskip

\noindent
We recall some notation.
Let $\Delta$ denote the coproduct
for a Hopf algebra $H$. Then the opposite coproduct $\Delta^{\rm op}$ 
is the composition
\begin{equation*}
\begin{CD} 
\Delta^{\rm op} : \quad H @>> \Delta >  
H\otimes H  
 @>> r\otimes s\mapsto s \otimes r > H\otimes H.
                  \end{CD}
\end{equation*}

\begin{definition}
\label{def:Lop}
\rm (See \cite[Section~9.1]{miki}.) 
Let $V$ denote a
$U_q(L({\mathfrak{sl}}_2))$-module and $0 \not=t\in \mathbb F$.
Consider an $\mathbb F$-linear map 
\begin{eqnarray*}
L:\quad 
V \otimes  {\bf V}(1,t)  \to 
V \otimes  {\bf V}(1,t).
\end{eqnarray*}
\noindent We call 
this map an  {\it $L$-operator for $V$ with parameter $t$} whenever 
the following diagram commutes for all $u \in
U_q(L({\mathfrak{sl}}_2))$:
\begin{equation*}
\begin{CD}
V\otimes {\bf V}(1,t) @>\Delta(u)>>   V \otimes {\bf V}(1,t)
	   \\ 
          @V L VV                   @VV L V \\
            V \otimes {\bf V}(1,t) @>> \Delta^{\rm op}(u) > 
                  V \otimes {\bf V}(1,t) 
		   \end{CD}
\end{equation*}
\end{definition}

\begin{definition}
\label{def:LoponV}
\rm (See \cite[Section~9.1]{miki}.) Let $V$ denote a
$U_q(L({\mathfrak{sl}}_2))$-module and $0 \not=t\in \mathbb F$.
Consider any
$\mathbb F$-linear map 
\begin{eqnarray}
\label{eq:Lop}
L:\quad 
V \otimes  {\bf V}(1,t)  \to 
V \otimes  {\bf V}(1,t). 
\end{eqnarray}
For $r,s \in \lbrace 0,1\rbrace$ define
an $\mathbb F$-linear map 
$L_{rs}: V \to V$ such that for $v \in V$,
\begin{eqnarray}
&&
L(v\otimes v_0) = L_{00}(v) \otimes v_0 + L_{10}(v) \otimes v_1,
\label{eq:L0}
\\
&&
L(v\otimes v_1) = L_{01}(v) \otimes v_0 + L_{11}(v) \otimes v_1.
\label{eq:L1}
\end{eqnarray}
Here $v_0, v_1$ is the basis for ${\bf V}(1,t)$
from Lemma
\ref{lem:irred}.
\end{definition}

\begin{lemma} 
\label{lem:Lopeq}
Referring to Definition
\ref{def:LoponV},
the map {\rm (\ref{eq:Lop})}
is an $L$-operator for $V$ with
parameter $t$
if and only if the following
equations hold on $V$:
\begin{eqnarray*}
&&
K_1 L_{00} = L_{00} K_1,  \qquad \qquad K_1 L_{01} = q^{-2} L_{01} K_1,
\\
&&
K_1 L_{10} = q^2 L_{10} K_1,  \qquad \qquad K_1 L_{11} = L_{11} K_1;
\\
\\
&&
L_{00} E_1 - q E_1 L_{00} = L_{10},
\qquad \qquad 
L_{01} E_1 - q E_1 L_{01} = L_{11} - L_{00} K_1,
\\
&&
L_{10} E_1 - q^{-1} E_1 L_{10} = 0,
\qquad \qquad 
L_{11} E_1 - q^{-1} E_1 L_{11} = - L_{10} K_1;
\\
\\
&&
F_1 L_{00}  - q^{-1}  L_{00}F_1 = L_{01},
\qquad \qquad 
F_1 L_{01}  - q  L_{01} F_1 = 0,
\\
&&
F_1 L_{10}  - q^{-1}  L_{10} F_1 = L_{11}-K_0 L_{00},
\qquad \qquad 
F_1 L_{11} - q L_{11} F_1 = - K_0 L_{01};
\\
\\
&&
K_0 L_{00} = L_{00} K_0,  \qquad \qquad K_0 L_{01} = q^{2} L_{01} K_0,
\\
&&
K_0 L_{10} = q^{-2} L_{10} K_0,  \qquad \qquad K_0 L_{11} = L_{11} K_0;
\\
\\
&&
L_{00} E_0 - q^{-1} E_0 L_{00} = -t L_{01}K_0,
\qquad \qquad 
L_{01} E_0 - q^{-1} E_0 L_{01} = 0,
\\
&&
L_{10} E_0 - q E_0 L_{10} = t L_{00}-t L_{11} K_0,
\qquad \qquad 
L_{11} E_0 - q E_0 L_{11} =  t L_{01};
\\
\\
&&
F_0 L_{00}  - q  L_{00}F_0 = -t^{-1} K_1 L_{10},
\qquad \qquad 
F_0 L_{01}  - q^{-1}  L_{01} F_0 =t^{-1} L_{00} -t^{-1} K_1 L_{11},
\\
&&
F_0 L_{10}  - q L_{10} F_0 =0,
\qquad \qquad 
F_0 L_{11} - q^{-1} L_{11} F_0= t^{-1} L_{10}.
\end{eqnarray*}
\end{lemma}
\begin{proof} This is routinely checked.
\end{proof}

\begin{example}
\label{ex:LopEval}
\rm (See \cite[Appendix]{kr},
\cite[Proposition~9.2]{miki}.)
Referring to Definition
\ref{def:LoponV}, assume that $V$ is an evaluation  module
${\bf V}(d,\mu)$ such that $q^{2i}\not=1$
for $1 \leq i \leq d$. Consider the matrices that
represent the
$L_{rs}$ with respect to the basis
$\lbrace v_i\rbrace_{i=0}^d$ for 
${\bf V}(d,\mu)$ from Lemma
\ref{lem:irred}.
Then the following are equivalent:
\begin{enumerate}
\item[\rm (i)] 
the map
 {\rm (\ref{eq:Lop}) }
is an $L$-operator for $V$ with parameter $t$;
\item[\rm (ii)]
the matrix entries are given in 
the table below (all matrix entries not shown are zero):
\bigskip

\centerline{
\begin{tabular}[t]{c|ccc}
{\rm operator} & {\rm $(i,i-1)$-entry} & {\rm $(i,i)$-entry} &
 {\rm $(i-1,i)$-entry}
   \\  \hline
$L_{00}$ &
$0$ & $ \frac{q^{1-i}-\mu^{-1}t q^{i-d}}{q-q^{-1}}\,\xi $     & $0$ 
  \\ 
$L_{01}$ &
 $\lbrack i \rbrack_q q^{1-i} \xi$ & $0$ & $0$
   \\
$L_{10}$ &
$0$ & $0$ &  $\lbrack d-i+1 \rbrack_q q^{i-d} \mu^{-1} t \xi $ 
   \\
$L_{11}$ &
$0$ & $ \frac{q^{i-d+1}-\mu^{-1}t q^{-i}}{q-q^{-1}}\,\xi   $   & $0$ 
     \end{tabular}
}
\medskip
\noindent Here $\xi \in \mathbb F$. 
\end{enumerate}
\end{example}

\begin{lemma}
\label{lem:LopUotimesV}
{\rm (See \cite[Proposition~9.3]{miki}.)}
Let $U$ and $V$ denote
$U_q(L({\mathfrak{sl}}_2))$-modules, and consider the
$U_q(L({\mathfrak{sl}}_2))$-module $U\otimes V$ from
Lemma
\ref{lem:UotimesV}. 
Let $0 \not=t\in \mathbb F$. Suppose we are given
$L$-operators for $U$ and $V$ with parameter $t$.
Then there exists an $L$-operator for $U\otimes V$ with parameter
$t$ such
that for $r,s\in \lbrace 0,1\rbrace$,
\begin{eqnarray}
L_{rs} (u \otimes v) = 
L_{r0}(u)\otimes L_{0s}(v) 
+
L_{r1}(u)\otimes L_{1s}(v) 
\qquad \qquad u\in U, \quad v \in V.
\label{eq:LL}
\end{eqnarray}
\end{lemma}
\begin{proof} For
 $r,s\in \lbrace 0,1\rbrace$ define
 an $\mathbb F$-linear map $L_{rs}: U\otimes V \to U \otimes V$
 that satisfies 
(\ref{eq:LL}). Using
(\ref{eq:LL})
and
Lemma
\ref{lem:UotimesV}
one checks that the $L_{rs}$ satisfy the equations in 
Lemma
\ref{lem:Lopeq}. The result follows by Lemma
\ref{lem:Lopeq}. 
\end{proof}

\begin{corollary} Adopt the notation and assumptions of
Proposition \ref{lem:CP}.
Then for $0 \not=t \in \mathbb F$
there exists a nonzero $L$-operator for $V$ with parameter $t$.
\end{corollary}
\begin{proof} By Proposition
\ref{lem:CP} along with
Example
\ref{ex:LopEval} and Lemma
\ref{lem:LopUotimesV}.
\end{proof}

\section{TD pairs and $L$-operators}

In Section 1 we discussed a TD pair $A,A^*$ on $V$.
We now return to this discussion, adopting the notation
and assumptions that were in force at the end of Section 1.
Recall the scalars $q,a,b$ from
 (\ref{eq:qracah}).
Recall the map $K$ from
above (\ref{eq:K}). 

\begin{proposition} 
\label{lem:TDuq}
{\rm (See \cite[p.~103]{ITaug}.)}
Assume that $\mathbb F$ is algebraically closed
with characteristic zero, and $q$ is not a root of
unity. Then
the vector space $V$ becomes a 
$U_q(L({\mathfrak{sl}}_2))$-module
on which $K=X_{31}$, $K^{-1}=X_{13}$ and
\begin{eqnarray*}
A = a X_{01}+a^{-1} X_{12},
\qquad \qquad 
A^* = b X_{23}+b^{-1} X_{30}.
\end{eqnarray*}
\end{proposition}
\begin{proof}  
This is how {\rm \cite[p.~103]{ITaug}} looks
from the equitable point of view.
\end{proof}

\begin{note}
\rm
The $U_q(L({\mathfrak{sl}}_2))$-module structure 
from Proposition
\ref{lem:TDuq} is not
 unique in general.
\end{note}

\noindent We now investigate the
$U_q(L({\mathfrak{sl}}_2))$-module structure 
from Proposition
\ref{lem:TDuq}.
Recall the map $R$ from
above (\ref{eq:AKR}).

\begin{lemma} 
\label{lem:Ruq} Assume that 
the vector space $V$ becomes a 
$U_q(L({\mathfrak{sl}}_2))$-module
on which $K=X_{31}$, $K^{-1}=X_{13}$ and
\begin{eqnarray*}
A = a X_{01}+a^{-1} X_{12},
\qquad \qquad 
A^* = b X_{23}+b^{-1} X_{30}.
\end{eqnarray*}
On this module,
\begin{enumerate}
\item[\rm (i)]
$R$ looks as follows in the
equitable presentation:
\begin{eqnarray}
R = a (X_{01}-X_{31}) + a^{-1} (X_{12}-X_{13}).
\label{eq:RXX}
\end{eqnarray}
\item[\rm (ii)] $R$ looks as follows in the
 Chevalley presentation:
\begin{eqnarray}
R = (q-q^{-1})(a q K_0 F_0 -a^{-1} E_1).
\label{eq:REF}
\end{eqnarray}
\end{enumerate}
\end{lemma}
\begin{proof} (i) 
In
line (\ref{eq:AKR}) eliminate $A,K,K^{-1}$ using the
assumptions of the present lemma.
\\
\noindent (ii) Evaluate the right-hand side of
(\ref{eq:RXX}) using the identifications from Lemma
\ref{lem:equitLoop}
and 
Note
\ref{note:ident}.
\end{proof}

\noindent We now present our main result.
Recall the Bockting operator $\psi$ from
(\ref{eq:psiActU}),
(\ref{eq:psiVSR}).

\begin{theorem}
\label{thm:main}
Assume that 
the vector space $V$ becomes a 
$U_q(L({\mathfrak{sl}}_2))$-module
on which $K=X_{31}$, $K^{-1}=X_{13}$ and
\begin{eqnarray*}
A = a X_{01}+a^{-1} X_{12},
\qquad \qquad 
A^* = b X_{23}+b^{-1} X_{30}.
\end{eqnarray*}
 Consider an $L$-operator for $V$ with parameter
$a^2$. Then 
 on $V$,
\begin{eqnarray}
\psi = -a (L_{00})^{-1} L_{01}
\label{eq:psiL}
\end{eqnarray}
provided that $L_{00}$ is invertible.
\end{theorem}
\begin{proof} Let $\widehat \psi$ denote the
expression on the right in
(\ref{eq:psiL}). We show
$\psi=\widehat \psi$. To do this, we show that 
$\widehat \psi$ satisfies
(\ref{eq:psiActU}), 
(\ref{eq:psiVSR}).
Concerning
(\ref{eq:psiActU}), 
by Lemma
\ref{lem:Lopeq} the equation
$K_0 {\widehat \psi} = q^2 
{\widehat \psi} K_0$ holds on $V$.
By Lemma
\ref{lem:equitLoop},
Note
\ref{note:ident}, and the construction, we obtain
$K_0 =X_{31}=K$ on $V$.
By these comments
$K {\widehat \psi} = q^2 
{\widehat \psi} K$ on $V$.
By this and
(\ref{eq:K}) we obtain
${\widehat \psi} U_i \subseteq U_{i-1}$ for $0 \leq i \leq d$.
So 
${\widehat \psi}$ satisfies
(\ref{eq:psiActU}).
Next we show that
${\widehat \psi}$ satisfies
(\ref{eq:psiVSR}).
Since $L_{00}$ is invertible and $K_0K_1=1$ it suffices to show that on $V$,
\begin{eqnarray}
\label{eq:L00W}
L_{00}({\widehat \psi} R - R {\widehat \psi})
= (q-q^{-1}) L_{00}(K_0-K_1).
\end{eqnarray}
By this and
(\ref{eq:REF}) it suffices to show that on $V$,
\begin{eqnarray}
a q L_{00}(
{\widehat \psi} K_0 F_0 - K_0 F_0 
{\widehat \psi}) - a^{-1}
L_{00}(
{\widehat \psi} E_1 - 
E_1 
{\widehat \psi} ) + L_{00}(K_1 - K_0)=0.
\label{eq:want}
\end{eqnarray}
We examine the terms in 
(\ref{eq:want}).
By Lemma
\ref{lem:Lopeq}
and the construction, the following hold on $V$:
\begin{eqnarray*}
L_{00} {\widehat \psi} K_0 F_0 &=& -a L_{01} K_0 F_0
\\
&=& -a q^{-2} K_0 L_{01} F_0  
\\
&=& -a q^{-1} K_0 ( F_0 L_{01}- a^{-2} L_{00} +  a^{-2} K_1 L_{11})
\end{eqnarray*}
and
\begin{eqnarray*}
L_{00} K_0 F_0 {\widehat \psi}
 &=& 
K_0 L_{00} F_0 {\widehat \psi}
\\
 &=&  q^{-1} K_0 (a^{-2} K_1 L_{10}+ F_0 L_{00}) {\widehat \psi}
\\
 &=&  q^{-1} K_0 ( 
 a^{-2} K_1 L_{10} {\widehat \psi}
 -a F_0 L_{01} 
 )
\end{eqnarray*}
and
\begin{eqnarray*}
L_{00} {\widehat \psi} E_1 &=& 
-a L_{01} E_1 
\\
&=&
-a (q E_1 L_{01} + L_{11}-L_{00} K_1) 
\\
&=&
-a (q E_1 L_{01} + L_{11}-K_1 L_{00}) 
\end{eqnarray*}
and
\begin{eqnarray*}
L_{00} E_1 {\widehat \psi} &=& 
(L_{10}+ q E_1 L_{00}) {\widehat \psi}
\\
&=& 
 L_{10} {\widehat \psi}
-q a E_1 L_{01} 
\end{eqnarray*}
and
\begin{eqnarray*}
L_{00}K_1 = K_1L_{00},
\qquad \qquad 
L_{00}K_0 = K_0L_{00}.
\end{eqnarray*}
To verify (\ref{eq:want}), evaluate
its left-hand side 
using the above comments and simplify
the result using 
$K_0K_1=1$.
The computation is routine, and omitted.
We have shown that
${\widehat \psi}$ satisfies (\ref{eq:psiVSR}).
The result follows.
\end{proof}

\section{Acknowledgment} The author thanks
Sarah Bockting-Conrad and Edward Hanson for giving this paper
a close reading and offering valuable suggestions.
The author also thanks
Pascal Baseilhac for many  conversations  
concerning quantum groups, $L$-operators, and tridiagonal pairs.

\noindent Paul Terwilliger \hfil\break
\noindent Department of Mathematics \hfil\break
\noindent University of Wisconsin \hfil\break
\noindent 480 Lincoln Drive \hfil\break
\noindent Madison, WI 53706-1388 USA \hfil\break
\noindent email: {\tt terwilli@math.wisc.edu }\hfil\break

\end{document}